\documentclass[twoside, 12pt]{amsart}
\usepackage{latexsym}
\usepackage{amssymb,amsmath,amsopn}
\usepackage[dvips]{graphicx}   
\usepackage{color,epsfig}      

\makeatletter
\def\@settitle{\begin{center}%
  \baselineskip14\p@\relax
  \bfseries
  \uppercasenonmath\@title
  \@title
  \ifx\@subtitle\@empty\else
     \\[1ex]\uppercasenonmath\@subtitle
     \footnotesize\mdseries\@subtitle
  \fi
  \end{center}%
}
\def\subtitle#1{\gdef\@subtitle{#1}}
\def\@subtitle{}
\makeatother

\newtheorem{theorem}{Theorem}

\newtheorem{problem}{Problem}

\DeclareMathOperator{\conv}{conv}

\begin{document}
\title[]{An extremal problem of regular simplices}
\subtitle{The higher-dimensional case}
\author[\'A. G.Horv\'ath]{\'Akos G.Horv\'ath}
\address {\'A. G.Horv\'ath \\ Department of Geometry \\ Mathematical Institute \\
Budapest University of Technology and Economics\\
H-1521 Budapest\\
Hungary}
\email{ghorvath@math.bme.hu}
\date{}


\begin{abstract}
The new result of this paper connected with the following problem: Consider a supporting hyperplane of a regular simplex and its reflected image at this hyperplane. When will be the volume of the convex hull of these two simplices maximal? We prove that in the case when the dimension is less or equal to $4$, the maximal volume achieves in that case when the hyperplane goes through on a vertex and orthogonal to the height of the simplex at this vertex. More interesting that in the higher dimensional cases this position is not optimal. We also determine the optimal position of hyperplane in the $5$-dimensional case. This corrects an erroneous statement in my paper \cite{gho 1}.
\end{abstract}

\maketitle

\section{Introduction}

Assume that the intersecting simplices $S$ and $S_H$ are reflected copies of each other in the hyperplane $H$. Then $H$ intersects each of them in the same set. By the Main Lemma of paper \cite{gho 1} we have that the intersection of the simplices in an optimal case (when the volume of their common convex hull is maximal) is a common vertex. In the paper \cite{gho 1} we stated the following (see Theorem 3 in \cite{gho 1}):
\emph{If $S$ is  the regular simplex of dimension $n$, then $c(S,S^H):=\frac{1}{\mathrm{ Vol }_{n}(S)}\mathrm{ Vol }(\mathrm{ conv }(S,S^H))=2n$}.
The proof contained a wrong substitution implying partly-false result. We prove that this statement is true when $n\leq 4$ and it is false in the higher dimensional cases. Additionally, we solve the $5$-dimensional case and obtain that position of the simplex which gives maximal volume. Since this position cannot be generalized trivially for higher dimensions, we have opened the following problem:

\begin{problem}
Consider a supporting hyperplane $H$ of a regular $n$-simplex $S$ and denote by $S'_H$ the orthogonal projection of $S$ to $H$. Determine that supporting hyperplane $H$ for which the volume of $\conv\{S,S'_H\}$ is maximal.
\end{problem}

As we mentioned Theorem \ref{thm:nlessorequaltofive} gives the answer in the cases when $n\leq 5$. We may mention here a very similar problem solved by P. Filliman. He investigated in \cite{filliman 1}, \cite{filliman 2} the volume of the projection of a body using the in the exterior algebra method. He determined those supporting hyperplanes $H$ of the regular simplex $S$ for which the volume of $S'_H$ is maximal or minimal, respectively. Similar interesting questions can be found in the paper \cite{gho-langi} on the volume of the union of two convex body and also in the survey paper \cite{gho-surveyonconvhull}.

\section{Regular simplex}

To avoid the confusion in our paper, we give a sort of easy calculations corresponding to the $n$-dimensional regular simplex.
Let us denote the vertices of the regular simplex of dimension $n$ by $\frac{1}{\sqrt{2}}e_i$, where $\{e_0,e_1,\ldots ,e_n\}$ is an orthonormed basis of an $(n+1)$-dimensional Euclidean space. Set $s_i:=\frac{1}{\sqrt{2}}(e_i-e_0)$ for $i=0,\ldots ,n$ the system of the vertices with respect to the $n$-dimensional hyperplane $\overline{H}:=\left\{\sum\limits_{i=0}^nx_ie_i \, : \, \sum\limits_{i=0}^nx_i=\frac{1}{\sqrt{2}}\right\}$. Then we have
$$
s=\frac{1}{\sqrt{2}}\sum\limits_{i=1}^ns_i=\frac{1}{\sqrt{2}}
\left(
\begin{array}{c}
-n \\
1\\
\vdots \\
1
\end{array}\right), \,
s_{2,k}=\frac{1}{\sqrt{2}}\sum\limits_{i=1}^{k-1}s_i=\frac{1}{\sqrt{2}}
\left(
\begin{array}{c}
-(k-1) \\
1\\
\vdots \\
1 \\
0 \\
\vdots \\
0
\end{array}\right)$$
$$
 s_{2,k}^0=\frac{ s_{2,k}}{\| s_{2,k}\|}, \, c=\frac{1}{n+1}s
$$
$$
\|s\|=\sqrt{\frac{n(n+1)}{2}}, \, \|s_{2,k}\|=\sqrt\frac{(k-1)k}{2},\, h=\frac{1}{n}\|s\|=\sqrt{\frac{n+1}{2n}}, \, \|c\|=\sqrt{\frac{n}{2(n+1)}}
$$
$$
u_0=\frac{s}{\|s\|}=\frac{1}{\sqrt{n(n+1)}}\left(
\begin{array}{c}
-n \\
1\\
\vdots \\
1
\end{array}\right),\,
u_i=\frac{c-s_i}{\|c-s_i\|}=\left(
\begin{array}{c}
1\\
\vdots \\
1\\
-n\\
1\\
\vdots \\
1
\end{array}\right)
$$
$$
\langle s_{2,k},u_0\rangle=\frac{(k-1)(n+1)}{\sqrt{2n(n+1)}} \, ,\, \langle u_i,u_j\rangle=-\frac{1}{n}
$$
$$
s-s_j=\left(
\begin{array}{c}
-n+1\\
1 \\
\vdots \\
1\\
0\\
1\\
\vdots \\
1
\end{array}\right), \,
s-(n+1)s_j=\left(
\begin{array}{c}
1\\
\vdots \\
1\\
-n\\
1\\
\vdots \\
1
\end{array}\right)=\sqrt{\frac{n(n+1)}{2}}u_j.
$$
Assume that the $n-1$-dimensional flat $H$ of the hyperplane $\overline{H}$ through the vertex $s_0$ contains precisely the vertices $s_0, \ldots s_r$ and parallel to the affine hull of the remaining ones. Then the unit normal vector of $H$ directed to the interior of that half-space $H^+$ which contains the simplex is
$$
u=\frac{1}{\sqrt{(n-r)(r+1)(n+1)}}\left(\sum\limits_{i=0}^r-(n-r)e_i+\sum\limits_{i=r+1}^{n+1}(r+1)e_i\right)
$$
$$
=\frac{1}{\sqrt{(n-r)(r+1)(n+1)}}\left(
\begin{array}{c}
-(n-r)\\
\vdots \\
-(n-r)\\
(r+1)\\
\vdots \\
r+1
\end{array}\right)
$$
and in this case we have
$$
\langle u, u_0\rangle= \sqrt{\frac{n-r}{r+1}}\frac{1}{\sqrt{n}}.
$$
Observe that $s_0$ always lies in $H$ while $s_n$ does not lie in it. Hence holds the inequality $0\leq r\leq n-1$.

\section{The theorem}

In this section we prove the following theorem:

\begin{theorem}\label{thm:nlessorequaltofive}
If $S$ is  the regular simplex of dimension $n\leq 4$, then
$$
c(S,S^H):=\frac{1}{\mathrm{ Vol }_{n}(S)}\mathrm{ Vol }(\mathrm{ conv }(S,S^H))=2n,
$$
attained only in the case when $u=u_0=\frac{s}{\|s\|}$. For $n=5$ we have
$$
c(S,S^H)=10\left(\frac{1}{2}+\frac{\sqrt{77}}{10\sqrt{3}}\right)\approx 10.06623 >10=2n.
$$
\end{theorem}

\begin{proof}
Without loss of generality we can assume that $\|s_1\|=\ldots =\|s_n\|=1$ as in the previous paragraph. We imagine that $H$ is horizontal and $H^+$ is the upper half-space. Define the \emph{upper side of $S$} as the collection of those facets in which a ray orthogonal to $H$ and terminated in a far point of $H^+$ is first intersecting $S$.  The volume of the convex hull is twice the sum of the volumes of those prisms which are based on the orthogonal projection of a facet of the simplex of the upper side. Let denote $F_{i_1},\cdots, F_{i_k}$ the simplex of the upper side, $F'_{i_1},\cdots, F'_{i_k}$ its orthogonal projections on $H$ and $u_{i_1},\cdots, u_{i_k}$ its respective unit normals, directed outwardly.
Observe that in this case
$$
u_{i_l}= \left\{
\begin{array}{lcc}
\frac{(n+1)s_{i_l}-s}{\| (n+1)s_{i_l}-s\|} & \mbox{if} & i_l\neq 0 \\
\frac{ s}{\| s\|} & \mbox{if} & i_l=0,
\end{array}
\right.
$$
moreover
$$
\| (n+1)s_{i_l}-s\|=|\langle u_{i_l}, (n+1)s_{i_l}-s\rangle|=|(n+1)\langle u_{i_l}, s_{i_l}\rangle-\langle u_{i_l}, s_{i_l}\rangle|=\sqrt{\frac{n(n+1)}{2}}\|s_1\|.
$$
We also introduced the notation $s=\sum\limits_{i=0}^{n}s_i=\sum\limits_{i=1}^{n}s_i$. Since $i_1=0$ corresponds to an upper facet, using Statement 2 in \cite{gho 1} we get
$$
\frac{\mathrm{ Vol }(\mathrm{ conv }(S,S^H))}{\mathrm{ Vol }_{n}(S)}=2n\sum\limits_{l=1}^k\frac{\langle u_{i_l}, u\rangle \langle u, s-s_{i_l}\rangle}{|\langle u_{i_l}, (n+1)s_{i_l}-s\rangle |}
$$
$$
=2n\left(\langle u_0, u\rangle ^2+ \frac{2}{(n+1)n}\sum\limits_{l=2}^k\langle  -(n+1)s_{i_l}+s, u \rangle\langle u, s-s_{i_l}\rangle\right)
$$
$$
=2n\left(\langle u_0, u\rangle ^2+ \frac{2}{(n+1)n}\sum\limits_{l=2}^k\left(-(n+1)\langle s_{i_l},u\rangle+\langle s, u \rangle\right)\left(\langle u, s\rangle-\langle u,s_{i_l}\rangle\right)\right)
$$
$$
=2n\left(\langle u_0, u\rangle ^2+ \frac{2}{(n+1)n}\sum\limits_{l=2}^k\left((n+1)\left\langle s_{i_l},u\right\rangle^2-(n+2)\langle s, u \rangle\left\langle u, s_{i_l}\right\rangle+\langle u,s\rangle^2\right)\right).
$$

If the only upper facet corresponds to the normal vector $u_0$, then only the first term occurs -- meaning that $k=1$ --- and the maximal value of the right hand side is less than or equal to $2n$ with equality in the required case.

Assume now that $k\geq 2$.
By the regularity of the simplex we have that $s=\sqrt{\frac{(n+1)n}{2}}u_0$, hence we get
$$
\frac{\mathrm{ Vol }(\mathrm{ conv }(S,S^H))}{\mathrm{ Vol }_{n}(S)}:=2nf\left(\left\langle s_{2,k}^0,u\right\rangle,\langle u_0, u \rangle\right)=
$$
$$
=2n\left(\langle u_0, u\rangle ^2+\sum\limits_{l=2}^k\left(\frac{2}{n}\left\langle s_{i_l},u\right\rangle^2-\sqrt{\frac{2}{n(n+1)}}(n+2)\langle u_0, u \rangle\left\langle  s_{i_l},u\right\rangle+\langle u_0,u\rangle^2\right)\right)
$$
$$
= 2n\left(\frac{2}{n}\sum\limits_{l=2}^k\left\langle s_{i_l},u\right\rangle^2-\sqrt{\frac{2}{n(n+1)}}(n+2)\langle u_0, u \rangle\left\langle  \sum\limits_{l=2}^k s_{i_l},u\right\rangle+k\langle u_0,u\rangle^2\right)
$$

Set  $s_{2,k}:=\sum\limits_{l=2}^ks_{i_l}$ and  $s_{2,k}^0:=\frac{s_{2,k}}{\|s_{2,k}\|}$, respectively.

Denote by $f$ the expression in the bracket then
$$
2nf:=2n\left(\frac{2}{n}\sum\limits_{l=2}^k\left\langle s_{i_l},u\right\rangle^2 -\sqrt{\frac{(k-1)k}{n(n+1)}}(n+2)\langle u_0, u \rangle\left\langle  s_{2,k}^0,u\right\rangle+k\langle u_0,u\rangle^2\right).
$$
First we remark that the inequalities $\frac{1}{n}\leq \langle u_0,u\rangle\leq 1$ are fulfilled.
We can observe that if a vertex $s_{i_l}$ gives an upper facet then
$$
\left\langle \sum\limits_{i\neq i_l}\left(s_i-s_{i_l}\right), u\right\rangle\geq 0,
$$
implying that
$$
\left\langle s-(n+1)s_{i_l}, u\right\rangle\geq 0.
$$
From this we get a new connection between the parameters $\left\langle s_{i_l},u\right\rangle$ and $\langle u_0,u\rangle$, namely
$$
\left\langle s_{i_l},u\right\rangle\leq \frac{\|s\|}{(n+1)}\langle u_0,u\rangle=\sqrt{\frac{n}{2(n+1)}}\langle u_0,u\rangle.
$$
This implies that
$$
\frac{2}{n}\sum\limits_{l=2}^k\left\langle s_{i_l},u\right\rangle^2\leq \frac{(k-1)}{(n+1)}\langle u_0,u\rangle^2.
$$
On the other hand, if we write that
$$
\langle u_0,u\rangle:=\cos \alpha  \mbox{, } \langle s_{2,k}^0,u_0\rangle :=\cos \beta \mbox{ and } \langle s_{2,k}^0,u\rangle:=\cos \gamma,
$$
then we get that $\gamma \leq \alpha + \beta$, and so $\cos \alpha \cos \beta-\sin\alpha \sin\beta \leq \cos \gamma$.
But
$$
\cos\beta =\frac{(k-1)+\frac{1}{2}(k-1)(n-1)}{\sqrt{\frac{(k-1)kn(n+1)}{4}}}=\sqrt{1-\frac{n-k+1}{nk}} \mbox{ and } \sin\beta=\sqrt{\frac{n-k+1}{nk}},
$$
hence we have a second inequality which is:
$$
\langle u_0,u\rangle\sqrt{\frac{(n+1)(k-1)}{nk}}-\sqrt{1-\langle u_0,u\rangle^2}\sqrt{\frac{n-k+1}{nk}}\leq  \langle s_{2,k}^0,u\rangle=\sqrt{\frac{2}{k(k-1)}}\sum\limits_{l=2}^k\left\langle s_{i_l},u\right\rangle.
$$
Introduce the notation $x:=\langle u_0,u\rangle$. Now we get that
$$
f(x)\leq
$$
$$
\left(\left( \frac{k-1}{n+1}+k\right)- \frac{(k-1)(n+2)}{n}\right)x^2+
\frac{(n+2)}{n}\sqrt{\frac{(k-1)(n-k+1)}{n+1}}x\sqrt{1-x^2}=
$$
$$
x\left[\left(1-(k-1)(n+2)\left(\frac{1}{n}-\frac{1}{n+1}\right)\right)x+\frac{(n+2)}{n}\sqrt{\frac{(k-1)(n-k+1)}{n+1}}\sqrt{1-x^2}\right]=
$$
$$
x\left[\left(1-\frac{(k-1)(n+2)}{n(n+1)}\right)x+\frac{(n+2)}{n}\sqrt{\frac{(k-1)(n-k+1)}{n+1}}\sqrt{1-x^2}\right]=
$$
$$
x\left[Ax+\sqrt{B}\sqrt{1-x^2}\right].
$$
We have to prove that for all $x\in [\frac{1}{n},1]$, $n\geq 3$ and $2\leq k\leq n$ holds the inequality
$$
x\left[Ax+\sqrt{B}\sqrt{1-x^2}\right]<1.
$$
We may assume that $x >0$. This inequality can be arranged to the form
$$
\frac{1}{x}- Ax > \sqrt{B}\sqrt{1-x^2}.
$$
Observe that the left hand side is greater than zero, because $0<A<1$ always hold. (This implies that $\frac{1}{x}- Ax>\frac{1}{x}- x>0$.) Considering the square of the inequality we get the following one:
$$
(A^2+B)x^4-(2A+B)x^2+1>0.
$$
The possible roots of the quadric in the left hand side is
$$
(x^2)_{1,2}=\frac{(2A+B)\pm \sqrt{B(4A+B-4)}}{2(A^2+B)}.
$$
These are real numbers if and only if $4A+B-4=B+4(A-1)\geq 0$. From this inequality we get
$$
0\leq \frac{(n+2)^2}{n^2}\frac{(k-1)(n-k+1)}{n+1}-4\frac{(k-1)(n+2)}{n(n+1)},
$$
equivalently
\begin{equation}\label{ineq:realroots}
4\leq \frac{(n-k+1)(n+2)}{n}.
\end{equation}

If $n=3$ than for $k=2,3$ the above inequality does not hold showing that the equation has no real solutions. This means that $f(x)<1$ and the statement is true.

If $n=4$ than for $k=3,4$ the inequality (\ref{ineq:realroots}) is false, too. However for $k=2$ it is hold, because $4<\frac{18}{4}$. In this case, $A=1-3/10=7/10$, $B=9/4\cdot 3/5=27/20$ and the roots are
$$
x^2_{1,2}=\frac{7/5+27/20\pm \sqrt{27/20(14/5+27/20-4)}}{2(49/100+27/20)}=\frac{11/4\pm 9/20}{92/25}
$$
$$
x_1^2=20/23
\qquad
x_2^2=5/8
$$
If the variable $x=\langle u_0,u\rangle$ lies between $\sqrt{5/8}$ and $\sqrt{20/23}$ then the examined function $f(x)$ can be greater or equal to $1$.

So we have to investigate the original formula in the following situation: Set $n=4$, $k=2$ and assume that $\sqrt{5/8}\leq\langle u_0,u\rangle\leq\sqrt{20/23}$.
Since $s_{2,k}^0=s_{i,2}$ we have the following inequalities between $\langle s_{i_2},u \rangle $ and $\langle u_0, u \rangle$,
$$
\sqrt{\frac{2}{5}}\langle u_0, u \rangle\geq \langle s_{i_2},u \rangle\geq \langle u_0, u \rangle\sqrt{\frac{5}{8}}-\sqrt{1-\langle u_0, u \rangle^2}\sqrt{\frac{3}{8}}
$$

The function
$$
f=\left(\frac{2}{n}\left\langle s_{i_2},u\right\rangle^2-\sqrt{\frac{2}{n(n+1)}}(n+2)\langle u_0, u \rangle\left\langle s_{i_2},u\right\rangle+k\langle u_0,u\rangle^2\right)
$$
is convex for a fixed value of $\langle s_{i_2},u \rangle$ hence it can takes its maximal values at the ends of its domain. Hence we have to determine the values of $f$ using the conditions $\sqrt{\frac{2}{5}}\langle u_0, u \rangle= \langle s_{i_2},u \rangle$ and $\langle s_{i_2},u \rangle= \langle u_0, u \rangle\sqrt{\frac{5}{8}}-\sqrt{1-\langle u_0, u \rangle^2}\sqrt{\frac{3}{8}}$, respectively. In the first case we get
$$
\left(\frac{1}{2}\left\langle s_{i_2},u\right\rangle^2 -3\sqrt{\frac{2}{5}}\langle u_0, u \rangle\left\langle  s_{i_2},u\right\rangle+2\langle u_0,u\rangle^2\right)=\left(\frac{1}{5}-3\frac{2}{5}+2\right)\langle u_0, u \rangle^2\leq 1,
$$
showing that $f(x)\leq 1$ as we stated. Secondly substitute the lower bound function to the expression of $f$. Using again the notation $x=\langle u_0, u \rangle$, we have to maximize the function
$$
g(x):=\frac{5}{8}x^2+\left(\sqrt{\frac{27}{20}}-\sqrt{\frac{15}{64}}\right)x\sqrt{1-x^2}+\frac{3}{16},
$$
on the interval $\sqrt{5/8}\leq x\leq \sqrt{20/23}$. It can be seen easily that on the interval $[1/2,\sqrt{20/23}]$ it is a concave function with an unique maximal value which is approximately $f(x)\approx 0.960977 $ attends at the value $x=0.915944<\sqrt{20/23}$, proving our statement.

Examine now the $5$-dimensional case. The inequality (\ref{ineq:realroots}) does not hold if $k>6-20/7=4-6/7>3$. Thus we have to investigate two respective cases, when $k=2$ or $k=3$. 

Let $k=3$. If we fixed the value of $x:=\langle u_0,u\rangle$ the function 
$$
f=\left(\frac{2}{5}\sum\limits_{l=2}^3\left\langle s_{i_l},u\right\rangle^2-\frac{7\sqrt{15}}{15}\langle u_0, u \rangle\left\langle  \sum\limits_{l=2}^3 s_{i_l},u\right\rangle+3\langle u_0,u\rangle^2\right)=
$$
$$
\left(\frac{2}{5}\left\langle s_{i_2},u\right\rangle^2-\frac{7\sqrt{15}}{15}\langle u_0, u \rangle\left\langle s_{i_2},u\right\rangle+\frac{3}{2}\langle u_0,u\rangle^2\right)+
$$
$$
\left(\frac{2}{5}\left\langle s_{i_3},u\right\rangle^2-\frac{7\sqrt{15}}{15}\langle u_0, u \rangle\left\langle s_{i_3},u\right\rangle+\frac{3}{2}\langle u_0,u\rangle^2\right)
$$
is the sum of two convex functions defined on the same interval. The maximal value of the two terms separately can be achieved only at the ends of the interval. We have an upper bound for $f$ if we determine the maximal value of the two terms separately and we add them. The left end of the examined interval gave with the equality $\left\langle s_{i_2},u\right\rangle=\langle u_0,u\rangle\sqrt{\frac{12}{5}}-\sqrt{1-\langle u_0,u\rangle^2}\sqrt{\frac{3}{5}}$ while the right end with the other one $\left\langle s_{i_2},u\right\rangle=\sqrt{\frac{5}{12}}\langle u_0,u\rangle$.
The sum of the two terms is less or equal to
$$
2\max\left\{\left(\frac{2}{5}\left(\langle u_0,u\rangle\sqrt{\frac{12}{5}}-\sqrt{1-\langle u_0,u\rangle^2}\sqrt{\frac{3}{5}}\right)^2-\right.\right.
$$
$$
\left. \left.\frac{7\sqrt{15}}{15}\langle u_0, u \rangle\left(\langle u_0,u\rangle\sqrt{\frac{12}{5}}-\sqrt{1-\langle u_0,u\rangle^2}\sqrt{\frac{3}{5}}\right)+\frac{3}{2}\langle u_0,u\rangle^2\right), \left(\frac{1}{6}-\frac{7}{6}+\frac{3}{2}\right)\langle u_0, u \rangle^2
\right\}=
$$
$$
2\max\left\{-\frac{29}{50}\langle u_0,u\rangle^2+\frac{11}{25}\langle u_0,u\rangle\sqrt{1-\langle u_0,u\rangle^2}+\frac{6}{25}, \frac{1}{2}\langle u_0, u \rangle^2
\right\}.
$$
The concave function
$$
h(x):=-\frac{29}{50}x^2+\frac{11}{25}x\sqrt{1-x^2}+\frac{6}{25}
$$
attends its maximal value $0.314005$ on the interval $[0,1]$ at the point $x=0.318833$ showing that $f(x)<1$ in this case, too. \footnote{In this case we did not have to use the smaller domain, based on the sharper calculation of the values $A=8/15$, $B=196/75$ and $(x^2)_{1,2}=\frac{9(23\pm 7\sqrt{2})}{326}$.}

The last case\footnote{In this case $A=23/30$, $B=98/75$ and $(x^2)_{1,2}=\frac{23/15+98/75\pm\sqrt{98/75(46/15+98/75-4)}}{2(23^2/30^2+98/75)}$.} is when $n=5$ and $k=2$.
The examined function $f$ is
$$
\frac{2}{5}\left\langle s_{i_2},u\right\rangle^2 -\frac{7\sqrt{15}}{15}\langle u_0, u \rangle\left\langle  s_{i,2},u\right\rangle+2\langle u_0,u\rangle^2
$$
and the conditions on the two variables are
$$
\langle u_0, u \rangle\sqrt{\frac{3}{5}}-\sqrt{1-\langle u_0, u \rangle^2}\sqrt{\frac{2}{5}}\leq \left\langle  s_{i,2},u\right\rangle\leq \sqrt{\frac{5}{12}}\langle u_0, u \rangle.
$$
Hence we get again that
$$
f(x)\leq \max\left\{\frac{17}{25}\langle u_0, u \rangle^2+\frac{23\sqrt{6}}{75}\langle u_0, u \rangle\sqrt{1-\langle u_0, u \rangle^2}+\frac{4}{25},\langle u_0,u\rangle^2\right\}.
$$
Consider the first argument of the right hand side and assume that $\left\langle  s_{i,2},u\right\rangle=\sin\varphi$ with a new variable $\varphi$. Then from the equality $ \sin \varphi:=\langle u_0, u \rangle\sqrt{\frac{3}{5}}-\sqrt{1-\langle u_0, u \rangle^2}\sqrt{\frac{2}{5}}$ we get that
$$
\langle u_0, u \rangle=\sqrt{\frac{3}{5}}\sin \varphi +\sqrt{\frac{2}{5}}\cos \varphi.
$$
The examined function now is
$$
g(\varphi)=\frac{17}{25}\left(\frac{3}{5}\sin ^2\varphi +\frac{2}{5}\cos^2 \varphi+ 2\frac{\sqrt{6}}{5}\sin\varphi\cos\varphi\right)+
$$
$$
\frac{23\sqrt{6}}{75}\left(\sqrt{\frac{3}{5}}\sin \varphi +\sqrt{\frac{2}{5}}\cos \varphi\right)\left(\sqrt{\frac{3}{5}}\cos \varphi -\sqrt{\frac{2}{5}}\sin \varphi\right)+\frac{4}{25}=
$$
$$
\left(\frac{51}{125}-6\frac{23}{5^33}\right)\sin ^2\varphi +\left(\frac{34}{125}+6\frac{23}{5^33}\right)\cos^2 \varphi+ \left(2\frac{17\sqrt{6}}{125}+\frac{23\sqrt{6}}{5^33}\right)\sin\varphi\cos\varphi+\frac{4}{25}=
$$
$$
\frac{1}{5}\sin^2\varphi +\frac{4}{5}\cos^2\varphi+\frac{\sqrt{6}}{3}\sin\varphi\cos\varphi=
\frac{3}{5}\cos^2\varphi+\frac{\sqrt{6}}{3}\sin\varphi\cos\varphi+\frac{1}{5}.
$$
This function takes its maximal value at that point $\varphi_{\max}$ for which $\cos^2\varphi_{\max}-\sin^2\varphi_{\max}=\sqrt{\frac{27}{77}}$ and $\sin\varphi_{\max}\cos\varphi_{\max}=\frac{1}{2}\sqrt{\frac{50}{77}}$, from which $\cos^2\varphi_{\max}=\frac{1}{2}\left(1+\sqrt{\frac{27}{77}}\right)$ and $\sin^2\varphi_{\max}=\frac{1}{2}\left(1-\sqrt{\frac{27}{77}}\right)$. The maximal value is
$$
g(\varphi_{\max})=\frac{1}{2}+\frac{3}{10}\sqrt{\frac{27}{77}}+\frac{10\sqrt{3}}{6\sqrt{77}}=\frac{1}{2}+\frac{\sqrt{77}}{10\sqrt{3}}\approx 1.006623.
$$
Since the corresponding values $\langle u_0, u \rangle$ and $\left\langle  s_{i,2},u\right\rangle$ are allowed in our investigation this proves that the getting upper bound can be achieved, too. This proves the statement.
\end{proof}

\begin{remark}
The optimal position of the regular simplex can be written geometrically in the five-dimensional case, too. We start with that position of the simplex, when the edge joining the vertices $s_0$ and $s_1$ lies in $H$ and the affine hull $G$ of the remaining vertices is parallel to $H$. The maximal volume can be get if we rotate this simplex around the orthogonal direct component $s_1^\bot$ of the line of $s_1$ with respect to $H$ by the angle $\varphi_{\max} $.
Before the rotation we have that
$$
s_0=\frac{1}{\sqrt{2}}\left(
\begin{array}{c}
1\\
0\\
0\\
0\\
0\\
0
\end{array}\right),
s_1=\frac{1}{\sqrt{2}}\left(
\begin{array}{c}
-1\\
1\\
0\\
0\\
0\\
0
\end{array}\right),
s_2=
\frac{1}{\sqrt{2}}\left(
\begin{array}{c}
-1\\
0\\
1\\
0\\
0\\
0
\end{array}\right),
s_3=
\frac{1}{\sqrt{2}}\left(
\begin{array}{c}
-1\\
0\\
0\\
1\\
0\\
0
\end{array}\right),
$$
$$
s_4=
\frac{1}{\sqrt{2}}\left(
\begin{array}{c}
-1\\
0\\
0\\
0\\
1\\
0
\end{array}\right),
s_5=
\frac{1}{\sqrt{2}}\left(
\begin{array}{c}
-1\\
0\\
0\\
0\\
0\\
1
\end{array}\right),
u=\frac{1}{2\sqrt{3}}\left(
\begin{array}{c}
-2\\
-2 \\
1\\
1\\
1\\
1
\end{array}\right),
u_0=
\frac{1}{\sqrt{30}}\left(
\begin{array}{c}
-5\\
1\\
1\\
1\\
1\\
1
\end{array}\right).
$$
Since $s_1$ is in that $2$-plane which is spanned by $u$ and $u_0$ we have the equality $u_0=\frac{\sqrt{15}}{5}s_1+\frac{\sqrt{10}
}{5}u$. Our rotation restricted to the $4$-subspace $\langle \{s_1^\bot, \sum_{i=1}^6e_i\}\rangle $ is the identity and on the $2$-subspace is generated by $s_1$ and $u$ it acts as a standard rotation. We consider the new orthonormal basis $\{\frac{1}{\sqrt{6}}\sum_{i=1}^6e_i,f_1,f_2,f_3,s_1,u\}$ where $\{f_1,f_2,f_3\}$ is an orthonormal basis of $s_1^\bot$. Clearly, $\langle\{s_5-s_2,s_5-s_3,s_5-s_4\}\rangle=\langle\{f_1,f_2,f_3\}\rangle$ hence we can choose $f_1$ to $s_5-s_2$, $f_2$ to $(s_5-s_3)-(s_5-s_4)=s_4-s_3$ and $f_3$ to $\frac{1}{\sqrt{2}}\left((-(s_5-s_2)+(s_5-s_3)+(s_5-s_4)\right)$, respectively. The orthogonal matrix of the basis change is
$$
B=\left(
\begin{array}{cccccc}
\frac{1}{\sqrt{6}}& 0 & 0 & 0 & -\frac{1}{\sqrt{2}} & -\frac{1}{\sqrt{3}}\\
\frac{1}{\sqrt{6}}& 0 & 0 & 0 & \frac{1}{\sqrt{2}} & -\frac{1}{\sqrt{3}}\\
\frac{1}{\sqrt{6}}& -\frac{1}{\sqrt{2}} & 0 & \frac{1}{2} & 0 & \frac{1}{2\sqrt{3}}\\
\frac{1}{\sqrt{6}}& 0 & -\frac{1}{\sqrt{2}} & -\frac{1}{2} &0 & \frac{1}{2\sqrt{3}}\\
\frac{1}{\sqrt{6}}& 0 & \frac{1}{\sqrt{2}} & -\frac{1}{2} & 0 & \frac{1}{2\sqrt{3}}\\
\frac{1}{\sqrt{6}}& \frac{1}{\sqrt{2}} & 0 & \frac{1}{2} & 0 & \frac{1}{2\sqrt{3}}
\end{array}\right).
$$
Since $B^{-1}=B^T$ we can get easily the new coordinates of the vertices. The new coordinates of the vertex  $\frac{1}{\sqrt{2}}e_i$ are the elements of the $i$-th row of $B$ multiply by $1/\sqrt{2}$, respectively. The translation of the origin to the first vertex $\frac{1}{\sqrt{2}}e_1$ of the simplex is equivalent to the subtraction of the first column of $\frac{1}{\sqrt{2}}B^T$ from the column vector of it, thus the new coordinates of the examined vectors are
$$
s_0=\left(
\begin{array}{c}
0\\
0\\
0\\
0\\
0\\
0
\end{array}\right),
s_1=\left(
\begin{array}{c}
0\\
0\\
0\\
0\\
1\\
0
\end{array}\right),
s_2=
\left(
\begin{array}{c}
0\\
-\frac{1}{2}\\
0\\
\frac{1}{\sqrt{8}}\\
\frac{1}{2}\\
\sqrt{\frac{3}{8}}
\end{array}\right),
s_3=
\left(
\begin{array}{c}
0\\
0\\
-\frac{1}{2}\\
-\frac{1}{\sqrt{8}}\\
\frac{1}{2}\\
\sqrt{\frac{3}{8}}
\end{array}\right),
$$
$$
s_4=
\left(
\begin{array}{c}
0\\
0\\
\frac{1}{2}\\
-\frac{1}{\sqrt{8}}\\
\frac{1}{2}\\
\sqrt{\frac{3}{8}}
\end{array}\right),
s_5=
\left(
\begin{array}{c}
0\\
\frac{1}{2}\\
0\\
\frac{1}{\sqrt{8}}\\
\frac{1}{2}\\
\sqrt{\frac{3}{8}}
\end{array}\right),
u_0=\left(
\begin{array}{c}
0\\
0\\
0\\
0\\
\frac{\sqrt{15}}{5}\\
\frac{\sqrt{10}}{5}
\end{array}\right),
u=
\left(
\begin{array}{c}
0\\
0\\
0\\
0\\
0\\
1
\end{array}\right).
$$
The matrix of the rotation and the respective rotated vertices are
$$
\left(
  \begin{array}{cccccc}
    1 & 0 & 0 & 0 & 0 & 0 \\
    0 & 1 & 0 & 0 & 0 & 0 \\
    0 & 0 & 1 & 0 & 0 & 0 \\
    0 & 0 & 0 & 1 & 0 & 0 \\
    0 & 0 & 0 & 0 & \cos\varphi & -\sin\varphi \\
    0 & 0 & 0 & 0 & \sin\varphi & \cos\varphi \\
  \end{array}
\right),
$$
$$
s_0(\varphi)=\left(
  \begin{array}{c}
    0  \\
    0  \\
    0  \\
    0  \\
    0  \\
    0  \\
  \end{array}
\right),
s_1(\varphi)=\left(
  \begin{array}{c}
    0  \\
    0  \\
    0  \\
    0  \\
    \cos\varphi  \\
    \sin\varphi  \\
  \end{array}
\right),
s_2(\varphi)=\left(
  \begin{array}{c}
   0\\
-\frac{1}{2}\\
0\\
\frac{1}{\sqrt{8}}\\
\frac{1}{2}\cos\varphi-\sqrt{\frac{3}{8}}\sin\varphi\\
\frac{1}{2}\sin\varphi+\sqrt{\frac{3}{8}}\cos\varphi
  \end{array}
\right),
$$
$$
s_3(\varphi)=\left(
\begin{array}{c}
0\\
0\\
-\frac{1}{2}\\
-\frac{1}{\sqrt{8}}\\
\frac{1}{2}\cos\varphi-\sqrt{\frac{3}{8}}\sin\varphi\\
\frac{1}{2}\sin\varphi+\sqrt{\frac{3}{8}}\cos\varphi
  \end{array}
\right),
s_4(\varphi)=
\left(
\begin{array}{c}
0\\
0\\
\frac{1}{2}\\
-\frac{1}{\sqrt{8}}\\
\frac{1}{2}\cos\varphi-\sqrt{\frac{3}{8}}\sin\varphi\\
\frac{1}{2}\sin\varphi+\sqrt{\frac{3}{8}}\cos\varphi
  \end{array}
\right),
$$
$$
s_5(\varphi)=
\left(
\begin{array}{c}
0\\
\frac{1}{2}\\
0\\
\frac{1}{\sqrt{8}}\\
\frac{1}{2}\cos\varphi-\sqrt{\frac{3}{8}}\sin\varphi\\
\frac{1}{2}\sin\varphi+\sqrt{\frac{3}{8}}\cos\varphi
  \end{array}
\right).
$$
Since we also have
$$
u_0(\varphi)=\left(
\begin{array}{c}
0\\
0\\
0\\
0\\
\frac{\sqrt{15}}{5}\cos\varphi-\frac{\sqrt{10}}{5}\sin\varphi\\
\frac{\sqrt{15}}{5}\sin\varphi+\frac{\sqrt{10}}{5}\cos\varphi
\end{array}\right) \qquad \mbox{ and }
u=\left(
\begin{array}{c}
0\\
0\\
0\\
0\\
0\\
1\\
\end{array}\right),
$$
we get the calculation of the previous paragraph showing that the maximal volume attends at $\cos \varphi_{\max}=\sqrt{\frac{1}{2}\left(1+\sqrt{\frac{27}{77}}\right)}$. Using the notation where $a:=\left(1+\sqrt{\frac{27}{77}}\right)$ and $b:=\left(1-\sqrt{\frac{27}{77}}\right)$, we get that the eleven vertices of the optimal polyhedron are
$$
\left(
  \begin{array}{c}
    0  \\
    0  \\
    0  \\
    0  \\
    0  \\
    0  \\
  \end{array}
\right),
\left(
  \begin{array}{c}
    0  \\
    0  \\
    0  \\
    0  \\
    \sqrt{\frac{1}{2}a}  \\
    \pm \sqrt{\frac{1}{2}b}  \\
  \end{array}
\right),
\left(
  \begin{array}{c}
   0\\
-\frac{1}{2}\\
0\\
\frac{1}{\sqrt{8}}\\
\sqrt{\frac{1}{8}a}-\sqrt{\frac{3}{16}b}\\
\pm\left(\sqrt{\frac{3}{16}a}+\sqrt{\frac{1}{8}b}\right)
  \end{array}
\right),
\left(
\begin{array}{c}
0\\
0\\
-\frac{1}{2}\\
-\frac{1}{\sqrt{8}}\\
\sqrt{\frac{1}{8}a}-\sqrt{\frac{3}{16}b}\\
\pm\left(\sqrt{\frac{3}{16}a}+\sqrt{\frac{1}{8}b}\right)
  \end{array}
\right),
$$
$$
\left(
\begin{array}{c}
0\\
0\\
\frac{1}{2}\\
-\frac{1}{\sqrt{8}}\\
\sqrt{\frac{1}{8}a}-\sqrt{\frac{3}{16}b}\\
\pm\left(\sqrt{\frac{3}{16}a}+\sqrt{\frac{1}{8}b}\right)
  \end{array}
\right),
\left(
\begin{array}{c}
0\\
\frac{1}{2}\\
0\\
\frac{1}{\sqrt{8}}\\
\sqrt{\frac{1}{8}a}-\sqrt{\frac{3}{16}b}\\
\pm\left(\sqrt{\frac{3}{16}a}+\sqrt{\frac{1}{8}b}\right)
  \end{array}
\right),
$$
respectively. The normals of the upper facets are
$$
u_0(\varphi_{\max})=\left(
\begin{array}{c}
0\\
0\\
0\\
0\\
\sqrt{\frac{3}{10}a}-\sqrt{\frac{2}{10}b}\\
\sqrt{\frac{2}{10}a}+\sqrt{\frac{3}{10}b}
  \end{array}
\right) \, \mbox{ and } \,
u_1(\varphi_{\max})=\left(
\begin{array}{c}
0\\
0\\
0\\
0\\
-\sqrt{\frac{3}{10}a}-\sqrt{\frac{2}{10}b}\\
\sqrt{\frac{2}{10}a}-\sqrt{\frac{3}{10}b}
  \end{array}
\right).
$$
We can imagine the corresponding body as the disjoint union of two pyramids and a simplex. The common base of the pyramids is a four-dimensional prism defined by the convex hull of the vertices $s_2(\varphi_{\max}),s_3(\varphi_{\max}),s_4(\varphi_{\max}),s_5(\varphi_{\max})$ and its reflected images $s'_2(\varphi_{\max}),s'_3(\varphi_{\max}),s'_4(\varphi_{\max}),s'_5(\varphi_{\max})$ with 4-dimensional volume
$$
\left(\frac{1}{6}\cdot 1\cdot 1\cdot 1\cdot \frac{1}{\sqrt{2}}\right)2\left(\sqrt{\frac{3}{16}a}+\sqrt{\frac{1}{8}b}\right)=\frac{1}{3}\left(\sqrt{\frac{3}{32}a}+\sqrt{\frac{1}{16}b}\right).
$$
 The apex of the first pyramid is $s_0$, while of the second one is $s_1(\varphi_{\max})$. From this we get that the volume of the union is
$$
\frac{1}{5}\sqrt{\frac{1}{2}a}\cdot\frac{1}{3}\left(\sqrt{\frac{3}{32}a}+\sqrt{\frac{1}{16}b}\right)=\frac{1}{3\cdot 5}\left(\sqrt{\frac{3}{64}}a+\sqrt{\frac{1}{32}ab}\right)=
$$
$$
\frac{1}{3\cdot 5}\left(\frac{\sqrt{3}}{8}\left(1+\sqrt{\frac{27}{77}}\right)+\sqrt{\frac{1}{32}\frac{50}{77}}\right)=\frac{1}{5!}\left(\sqrt{3}+\frac{19}{\sqrt{77}}\right).
$$
The last part is that simplex which defined by the vertices $\pm s_1(\varphi_{\max})$, $s'_2(\varphi_{\max})$, $s'_3(\varphi_{\max})$, $s'_4(\varphi_{\max})$, $s'_5(\varphi_{\max})$. Its volume is $\frac{1}{4\cdot 5!}\left(\sqrt{3}+\frac{1}{\sqrt{77}}\right)$ giving the volume of the body:
$$
v=\frac{1}{4\cdot 5!}\left(5\sqrt{3}+\sqrt{77}\right).
$$
Since the volume of the regular simplex of dimension $5$ with edge length $1$ is $\frac{\sqrt{3}}{4\cdot 5!}$ we get immediately again the ratio $10\left(\frac{1}{2}+\frac{\sqrt{77}}{10\sqrt{3}}\right)$ of the previous paragraph.

\end{remark}

\section{Acknowledgement}

I am thankful to my colleague Zsolt L\'angi who found the mistake in the proof of Theorem 3 of paper \cite{gho 1} and inspired me to write the present ones. I also thanks for the helpful suggestions of Hans Havlicek.

\end{document}